\documentclass{amsart}

\usepackage{color,graphicx,amssymb,latexsym,amsfonts,txfonts,amsmath,amsthm}
\usepackage{pdfsync,enumitem}
\usepackage{amsmath,amscd}
\usepackage[all,cmtip]{xy}

\usepackage{tikz}
\usetikzlibrary{matrix}

\input epsf

\usepackage{hyperref}
\hypersetup{
    colorlinks=true,       % false: boxed links; true: colored links
    linkcolor=blue,          % color of internal links
    citecolor=blue,        % color of links to bibliography
    filecolor=blue,      % color of file links
    urlcolor=blue           % color of external links
}

\newtheorem{theo}{Theorem}
\newtheorem{coro}{Corollary}
\newtheorem{prop}{Proposition}
\newtheorem{lemm}{Lemma}

\theoremstyle{remark}
\newtheorem{rema}{\bf Remark}

\begin{document}

\title{Extending finite free actions of surfaces}

%\author{Angel Carocca}
\author{Rub\'en A. Hidalgo}
\address{Departamento de Matem\'atica y Estad\'{\i}stica, Universidad de La Frontera. Temuco, Chile}
%\email{angel.carocca@ufrontera.cl}
\email{ruben.hidalgo@ufrontera.cl}

\thanks{Partially supported by Project Fondecyt 1230001}


%%%%%%%%%%%%%%%%%
%%%%%%%%%%%%%%%%%

\begin{abstract}
We prove the existence of finite groups acting freely as orientation-preserving homeomorphisms on closed orientable surfaces which extend as a group of homeomorphisms of some compact orientable $3$-manifold but which cannot extend to a handlebody.  This solves a basic problem in low-dimensional equivariant topology going back to the work of Reni and Zimmermann in the mid 1990s

\end{abstract}

\maketitle

%%%%%%%%%%%%%%%%%
%%%%%%%%%%%%%%%%%
\section{Introduction}
Every closed orientable surface $S$ of genus $g$ is null-cobordant, that is, it can be seen as the boundary of some compact orientable $3$-manifold (for instance, as the boundary of a handlebody of genus $g$). Let $G$ be a finite group of homeomorphisms of $S$. One says that $G$ {\it extends} if it is possible to find some compact orientable $3$-manifold $M^{3}$ with boundary $S$ admitting a group of homeomorphisms, isomorphic to $G$, whose restriction to $S$ coincides with $G$. In the case that $M^{3}$ can be chosen to be a handlebody, then we say that $G$ {\it extends to a handlebody}. A natural question is if every extendable action necessarily extends to a handlebody (it seems this question is due to B. Zimmermann and M. Reni).

If $g=0$, then $G$ always extends to the closed $3$-ball (a handlebody of genus zero). 
Let us assume $g \geq 1$. A {\it Schottky system of loops} for $G$ is a collection ${\mathcal F}$, of pairwise disjoint essential simple loops on $S$, such that: (i) ${\mathcal F}$ is $G$-invariant and (ii) $S \setminus {\mathcal F}$ consists of planar surfaces.  As a consequence of the equivariant loop theorem \cite{M-Y}, if $G$ extends to a handlebody, then there exists a Schottky system of loops for $G$. The converse is also true, that is, $G$ extends to a handlebody if and only if a Schottky system of loops exists for $G$ (see \cite{H-M} for proof of this fact in terms of Kleinian groups).

If $G$ consists only of orientation-preserving homeomorphisms, then the following is known (see, for instance, \cite{H1, H2,H4, RZ}):
(i) if $S/G$ has genus zero and exactly three cone points, then it cannot be extended to a handlebody,
(ii) if $G$ acts freely (i.e., the $G$-stabilizer of every point of $S$ is trivial) and $S/G$ has genus one, then it extends to a handlebody,
(iii) if $G$ acts freely and it is isomorphic to either an abelian group or one of the Platonic symmetry groups, then it extends to a handlebody,
(iv) if $G$ is a dihedral group, then it extends to a handlebody.
In \cite{CH}, is studied the case when $G$ is a cyclic group generated by an orientation-reversing homeomorphism. 

In \cite{GZ}, it was proved that there are Hurwitz actions of $G={\rm PSL}_{2}(q)$ (i.e., $G$ consists of orientation-preserving homeomorphisms and the quotient orbifold $S/G$ has genus zero and exactly three cone points of orders $2$, $3$ and $7$) which extend. As these actions cannot extend to a handlebody (as noted above), these provide examples that answer negatively the above extension question under the presence of fixed points. 
 
In this paper, for the case of free actions, we observe that there are free actions that extend but not to a handlebody (Theorem \ref{main2} and Corollary \ref{main2}), so again providing a negative answer to the above question. 
 
The main ideas are described below.
Let us assume that $G$ acts freely by orientation-preserving homeomorphisms on $S$ of genus $g \geq 2$ and let $R=S/G$ of genus $\gamma \geq 2$. 
The free action of $G$ on $S$ is defined by a surjective homomorphism $\theta: F \to S$, where $F=\pi_{1}(R)$.

In \cite{DS}, Dominguez and Segovia proved that if $G$ is either the alternating or the symmetric group, then it always extends. Note that these groups have the property that their Bogomolov multiplier $B_{0}(G)$ is zero. Generalizing these examples, in \cite{Samperton}, Samperton proved that if $B_{0}(G)=0$, then any free action of $G$ extends. In the same paper, it was proved that if $B_{0}(G) \neq 0$ and $G$ does not contain a dihedral subgroup, then it admits a free action that does not extend (for instance, $G={\rm SmallGroup}(3^{5},28)$ in the Gap Library \cite{GAP}). 

As a consequence of the above, if $B_{0}(G)=0$ and there is not a Schottky system of loops for $G$, then the free action of $G$ does extend but not to a handlebody, providing in this way the existence of free actions as desired.

One can use the HAP package, implemented in GAP \cite{GAP}, to verify if the Bogomolov multiplier of group $G$ is zero (in Section \ref{semidirecto}, we recall some general results to guarantee that the Bogomolov multiplier is zero for a certain class of groups). 
If $G$ has either (i) odd order or (ii) even order but with a unique element of order two, the existence or non-existence of a Schottky system of loops for $G$ can be read from $\theta$ (see Theorem \ref{main0} for $\gamma=2$ and Theorem \ref{main00} for $\gamma \geq 3$). 
If $\gamma=2$,  in which case, $F=\langle x_{1},y_{1},x_{2},y_{2}: [x_{1},y_{1}][x_{2},y_{2}]=1 \rangle$, then this existence result reads as follows.
If ${\mathfrak C}$ is the orbit of the commutator $[x_{1},y_{1}]$ under certain (explicit) subgroup ${\rm Out}_{0}^{+}(F)$ of ${\rm Aut}^{+}(F)$ (see Section \ref{geometrico}), then 
$G$ admits a Schottky system of loops (i.e., it extends to a handlebody) if and only if $\ker(\theta) \cap {\mathfrak C} \neq \emptyset$.
The collection ${\mathfrak C}$ is infinite, but, as $F$ contains a finite number of subgroups of a fixed finite index, there is a finite subcollection ${\mathfrak C}_{G}$ of ${\mathfrak C}$  such that $\ker(\theta) \cap {\mathfrak C} \neq \emptyset$ if and only if $\ker(\theta) \cap {\mathfrak C}_{G} \neq \emptyset$. To explicitly describe ${\mathfrak C}_{G}$ for a given $G$ is not a simple task, and we do not attempt to do so in this paper.

So, let us now assume that $G$ is a non-abelian group of odd order and $[x_{1},y_{1}] \notin \ker(\theta)$ (there are plenty of examples with this property; examples are provided in Section \ref{Sec:ejemplos}). 

If $N$ is the intersection of all the ${\rm Aut}^{+}(F)$-images of $\ker(\theta)$ (a finite index normal subgroup of $F$), then $N \cap {\mathfrak C}=\emptyset$ (see Corollary \ref{estrategia}). So, by Theorem \ref{main0}, this normal subgroup $N$ provides us with a free action of  $G_{N}=F/N$ as a group of orientation-preserving homeomorphisms of a closed orientable surface $S_{N}$ such that $S_{N}/G_{N}=R$ and which does not extend to a handlebody (note that there is a normal subgroup $H<G_{N}$ with $S=S_{N}/H$ and $G=G_{N}/H$).
Two situations may happen at this point:
either (i) $G_{N}$ does not extend (so, providing more examples as in \cite{Samperton}) or (ii) $G_{N}$ extends but it does not extend to a handlebody (providing a negative answer to Zimmermann's question). We interpret this construction as a fiber product to see that, if we also assume that all of the Sylow subgroups of $G$ are abelian, then 
$B_{0}(G_{N})=0$, providing in this way examples as required.

The author was privately communicated by E. Samperton that, together with M. Boggi and C. Segovia,  they may obtain examples of free actions that extend but not to a handlebody by using different methods \cite{BSS}.

%%%%%%%%%%%%%%%%%%
%%%%%%%%%%%%%%%%%%
\section{Preliminaries}
%%%%%%%%%%%%%
\subsection{Actions on handlebodies and Schottky groups}
A Schottky group $\Gamma$ of rank $g$ is a purely loxodromic Kleinian group, with a non-empty region of discontinuity $\Omega \subset \widehat{\mathbb C}$, isomorphic to the free group of rank $g$. It is known that $\Omega$ is connected (equals to $\widehat{\mathbb C}$ if $g=0$, $\widehat{\mathbb C}$ minus two points if $g=1$, and the complement of a Cantor set if $g \geq 2$). The quotient $\Omega/\Gamma$ is a closed Riemann surface of genus $g$ and $({\mathbb H}^{3} \cup \Omega)/\Gamma$ is a handlebody whose interior carries a complete hyperbolic metric (with injectivity radius bounded away from zero) whose conformal boundary is $\Omega/\Gamma$. 
Koebe's retrosection theorem asserts that every closed Riemann surface can be obtained, up to biholomorphisms, in this way \cite{Koebe}.

Let $G$ be a finite group of homeomorphisms of a closed orientable surface $S$ of genus $g\geq 2$. By the Nielsen realization theorem \cite{Kerckhoff}, we may provide to $S$ of a Riemann surface structure making $G$ a group of conformal/anticonformal automorphisms of it. Let us fix one of such a Riemann surface structures on $S$. If $M^{3}$ is a handlebody whose boundary is $S$, then the given Riemann surface structure on $S$ induces to the interior of $M^{3}$ a complete hyperbolic structure (with $S$ as its conformal boundary). This is equivalent to have a Schottky group $\Gamma$ of rank $g$, with region of discontinuity $\Omega \subset \widehat{\mathbb C}$, such that $S=\Omega/\Gamma$ and $M^{3}=({\mathbb H}^{3} \cup \Omega)/\Gamma$.  To say that $G$ extends to the handlebody $M^{3}$ is, in this setting, equivalent to the existence of a Kleinian group $K$ containing $\Gamma$ as a finite index normal subgroup such that the action of $G$ is represented by the quotient group $K/\Gamma$. In this case, one says that $K$ is a {\it virtual (extended) Schottky group} (if $G$ does not contain orientation reversing homeomorphisms, then we say that $K$ is a {\it virtual Schottky group}). The finite index condition asserts that $K$ and $\Gamma$ both have the same region of discontinuity. A geometrical picture of these types of groups, in terms of the Klein-Maskit combination theorems \cite{Maskit:Comb}, was provided in \cite{H3}. If $G$ acts freely on $S$ and contains no dihedral subgroups, such a picture is quite simple and is given as follows.

\begin{theo}[\cite{H3}]\label{picturevirtual}
Let $K$ be a virtual (extended) Schottky group containing as a finite index normal subgroup a Schottky group $\Gamma$ of rank $g \geq 2$. Assume that the group $K/\Gamma$ does not contain dihedral subgroups and that it acts freely on $\Omega/\Gamma$, where $\Omega$ is the region of discontinuity. Then $K$ is the free product, in the sense of the Klein-Maskit combination theorem, of 
(i) $\alpha \geq 0$ cyclic groups generated by loxodrmic elements $A_{j}$, 
(ii) $\alpha' \geq 0$ cyclic groups generated by pseudo-reflection elements $B_{j}$, 
(iii) $\beta \geq 0$ abelian groups, each one generated by an elliptic element $E_{j}$ (of some finite order $n_{j} \geq 2$) together a loxodromic element $C_{j}$ such that $C_{j}E_{j}C_{j}^{-1}=E_{j}$, and
(iv) $\beta' \geq 0$ groups, each one generated by an elliptic element $F_{j}$ (of some finite order $m_{j} \geq 2$) together a pseudo-reflection $D_{j}$ such that $D_{j}F_{j}D_{j}^{-1}=F_{j}^{-1}$.  
We say that $K$ has signature $(\alpha,\alpha',\beta,\beta')$.
\end{theo}

\begin{rema}
If the virtual (extended) Schottky group $K$ has signature $(\alpha,\alpha',\beta,\beta')$, then $\Omega/K$ is a closed surface being the connected sum of $\alpha+\beta$ tori and $2(\alpha'+\beta')$ real projective planes ($K$ is virtual Schottky group if and only if $\alpha'=\beta'=0$). If $\Gamma$ is a Schottky group of rank $g$, being a finite index normal subgroup of $K$, then the free action of $G=K/\Gamma$ on $S=\Omega/\Gamma$ extends to the handlebody $M^{3}=({\mathbb H}^{3} \cup \Omega)/\Gamma$ and $M^{3}/G=({\mathbb H}^{3} \cup \Omega)/K$ is topologically a handlebody and its conical locus (in its interior) consist of $\beta+\beta'$ pairwise disjoint (unlinked) simple loops (their cone orders being $n_{1},\ldots,n_{\beta}, m_{1},\ldots,m_{\beta'}$).
\end{rema}

%%%%%%%%%%%%%%%
\subsection{The Bogomolov multiplier}\label{semidirecto}
If $G$ is a finite group, then its Schur multiplier is the abelian group $MG):={\rm H}^{2}(G,{\mathbb Q}/{\mathbb Z}) \cong {\rm H}_{2}(G,{\mathbb Z})$ and its Bogomolov multiplier (see \cite{Bogomolov}) is the abelian group $B_{0}=\ker\left[{\rm H}^{2}(G,{\mathbb Q}/{\mathbb Z}) \to  \bigoplus_{A \subset G} {\rm H}^{2}(A,{\mathbb Q}/{\mathbb Z})\right]$, where $A$ runs over all abelian subgroups of $G$. If $M_{0}(G)$ is the subgroup of ${\rm H}_{2}(G,{\mathbb Z})$ generated by the toral classes of $M(G)$, then $B_{0}(G) \cong {\rm H}_{2}(G,{\mathbb Z})/M_{0}(G)$.

In \cite{Saltman}, Saltman observed that finite groups with non-zero Bogomolov multiplier provide counterexamples of Noether’s problem to the rationality of fields of invariants. In \cite{Bogomolov}, Bogomolov provided examples of $p$-groups of order $p^{9}$ with a non-zero Bogomolov multiplier.

Let $G=F_{n}/N$, where $F_{n}$ denotes the free group of rank $n$ and $N$ is a finite index normal subgroup of $F_{n}$.
In \cite{Hopf}, Hopf proved that $$M(G) \cong \frac{[F_{n},F_{n}] \cap N}{[F_{n},N]},$$
where $[F_{n},F_{n}]$ is the derived subgroup of $F_{n}$ and
$[F_{n},N]$ is the normal subgroup generated by the elements of the form $aba^{-1}b^{-1}$, $a \in F_{n}$ and $b \in N$, and later, in \cite{Moravec} (see also \cite{Sumana}), Moravec obtained a Hopf-type formula for computing $B_{0}(G)$
$$B_{0}(G) \cong \frac{[F_{n},F_{n}] \cap N}{\langle K(F_{n}) \cap N\rangle},$$
where  $K(F_{n})$ is the set of all the commutators of $F_{n}$.

The HAP package ({\rm http:// hamilton.nuigalway.ie/Hap/www/}), implemented in GAP \cite{GAP}, permits to compute the Bogomolov multiplier of a group $G$ of small order
(using the command BogomolovMultiplier($G$)).
The following facts can be found in, for instance, \cite{Kang,Kang1,Kuy,Michailov1}.
\begin{enumerate}[leftmargin=15pt]
\item $B_{0}(G)=0$ if $G$ is one of the followings:
\begin{enumerate}
\item a symmetric group;
\item a simple group;
\item a $p$-group of order at most $p^{4}$;
\item an abelian-by-cyclic groups (i.e., $G$ contains an abelian group $A$ as a normal subgroup such that $G/A$ is a cyclic group);
\item a primitive supersolvable group;
\item an extraspecial $p$-group (i.e., the center $Z_{G}$ is a cyclic group of order $p$ and $G/Z_{G} \cong {\mathbb Z}_{p}^{2n}$).
\end{enumerate}
\item $B_{0}(G_{1} \times G_{2}) \cong B_{0}(G_{1}) \times B_{0}(G_{2})$.
\item $B_{0}(N \rtimes K) \cong B_{0}(N)^{K} \times B_{0}(K)$ when ${\rm gcd}(|N|,|K|)=1$.
\end{enumerate}

The following result, which will be needed for some of our examples, is a consequence of \cite[Lemma 2.6.]{BMP} (see also \cite[Corollary 2.7]{BMP}).
\begin{lemm}\label{mainlemma0}
Let $G$ be a finite group such that all of its Sylow subgroups have zero Bogomolov multiplier. Then $B_{0}(G)=0$.
\end{lemm}

%%%%%%%%%%%%%%%%
\subsection{Samperton's theorem}
Each finite group $G$ can be realized as a group of orientation-preserving homeomorphisms of some closed orientable surface that acts freely. 
If the free action of $G$ on $S$ extends to a compact manifold $M$, then it might happen that in such extension the action is no longer free. In that case, if the $G$-stabilizer of every point is cyclic, then one says that the extension is non-singular (and that $G$ extends non-singularly).

In \cite{Samperton}, Samperton observed that $B_{0}(G)$ provides an obstruction for the extendability of (all possible) free actions of $G$.

Let us denote by $D_{n}$, where $n \geq 2$, the dihedral group of order $2n$, by ${\mathcal A}_{4}$ and ${\mathcal A}_{5}$ the alternating groups of orders $12$ and $60$, respectively, and by  ${\mathcal S}_{4}$ the symmetric group in $4$ letters.

\begin{theo}[Samperton \cite{Samperton}]\label{teo1}
Let $G$ be a finite group. Then
\begin{enumerate}[leftmargin=15pt]
\item  Every free action of $G$ on a closed orientable surface extends non-singularly if and only if $B_{0}(G)=0$.  
\item If $B_{0}(G) \neq 0$ and $G$ does not contain a subgroup isomorphic to either $D_{n}$, $n \geq 2$, ${\mathcal A}_{4}$, ${\mathcal A}_{5}$ or ${\mathcal S}_{4}$, then $G$ affords free actions on some closed orientable surface that do not extend.
\end{enumerate}
\end{theo}

A consequence, of part (1) of the above theorem, is the following fact (this, in particular, provides the results in \cite{DS})

\begin{coro}
Every free action of a group $G$, of either type (a)-(f) in Section \ref{semidirecto}, extends. 
\end{coro}

\begin{rema}
(i) The condition that $G$  does not contain a subgroup isomorphic to either $D_{n}$, $n \geq 2$, ${\mathcal A}_{4}$, ${\mathcal A}_{5}$ or ${\mathcal S}_{4}$, in part (2) of the above theorem, is equivalent for it to have at most one element of order two  (if a finite group has two different elements of order two, then these two generate a dihedral group). In particular, this holds if $G$ has odd order.
(ii) If $B_{0}(G) \neq 0$, then there might be both, extendable and non-extendable, free actions of $G$ on surfaces.
In \cite{Samperton}, there is observed that the group $G={\rm SmallGroup}(3^{5},28) \cong ({\mathbb Z}_{9} \rtimes {\mathbb Z}_{9}) \rtimes {\mathbb Z}_{3}$ satisfies  
(2) in Theorem \ref{teo1} ($M(G) \cong {\mathbb Z}_{9}$ and $B_{0}(G) \cong {\mathbb Z}_{3}$), so there exists a free action of it that cannot extend. In Section \ref{Sec:ejemplos}, we consider some free actions of this group and observe that some of them extend to a handlebody.

 \end{rema}

%%%%%%%%%%%%%%%%%%%%%
\subsection{Examples of groups which extends to handlebodies}
As observed in the previous section, if the group $G$ is either (i) abelian, (ii) dihedral, (iii) alternating, (iv) symmetric, or (v) abelian-by-cyclic, then their free actions always extend. In the case of abelian and dihedral groups, the free actions always extend to a handlebody. Below (see Proposition \ref{proposemidirecto}), we observe that the same situation happens for abelian-by-cyclic groups. (In Section \ref{Sec:ejemplos}, we provide other examples of free actions that extend to handlebodies).

Let us first recall that there is exactly one topological free action of a finite cyclic group on a closed orientable surface.  Let $h:S \to S$ be an order $q$ orientation-preserving homeomorphism of a closed Riemantable surface $S$ of genus $g \geq 2$ and $H=\langle h \rangle \cong {\mathbb Z}_{q}$ acting freely. Let $P:S \to R=S/H$ be a Galois cover with deck group $H$. Let $\delta \subset R$ be 
an essential dividing simple loop $\delta \subset R$ such that one of the components $T$ of $R \setminus \{\delta\}$ is of genus one. As $\delta$ is a commutator of the fundamental group of $R$, the collection $P^{-1}(\delta)$ consists of $q$ essential simple loops, each one with a trivial $H$-stabilizer. If $\alpha$ is one of the loops in $P^{-1}(\delta)$, then  either: 
(i) it is the only boundary of some component of $S \setminus P^{-1}(\delta)$ (in particular a commutator) or 
(ii) it is one of the $q$ boundaries of the component of $S \setminus P^{-1}(\delta)$ of genus one (so a product of commutators).

\begin{prop}\label{proposemidirecto}
Let $G$ be a finite group, admitting a normal abelian subgroup $A$ such that $C=G/A$ is a cyclic group. If $G$ acts freely as a group of orientation-preserving homeomorphisms of a closed orientable surface $S$ of genus $g \geq 2$, then it extends to a handlebody.
\end{prop}
\begin{proof}
Let $\pi:S \to R=S/G$ be a Galois covering with deck group $G$ and let $g_{R} \geq 2$ be the genus of $R$.
The surface $X=S/A$ is a closed orientable surface on which the induced group of orientation preserving homeomorphisms $C=G/K \cong {\mathbb Z}_{q}$ is acting freely and $X/C=R$. Let $P:S \to X$ and $Q:X \to R$ be Galois coverings, with respective deck groups $A$ and $C$, such that $\pi=Q \circ P$.

Let $\delta^{1},\ldots,\delta^{g_{R}} \subset R$ be a collection of pairwise disjoint essential dividing simple loops, each $\delta^{j}$ cutting $R$ into a surface $T_{j}$ of genus one (and other of genus $g_{R}-1$). 

As $Q$ has cyclic deck group, the collection $Q^{-1}(\delta^{j}) \subset X$ consists of exactly $q$ pairwise disjoint essential simple loops, $\delta^{j}_{1},\ldots,\delta^{j}_{q}$, each one being represented by a product of commutator elements of the fundamental group of $X$. Moreover, the collection of loops $\cup_{j=1}^{g_{R}} Q^{-1}(\delta^{j})$ divides $X$ into genus one and genus zero surfaces.

Now, as $P$ is an abelian Galois cover, with deck group the abelian group $A$ (and as each $\delta^{j}_{i}$ is in the commutator subgroup), it can be seen that $Q^{-1}(\delta^{j}_{i})$ consists of exactly $|A|$ simple loops (each with trivial $G$-stabilizer). Moreover, 
the collection of simple lops ${\mathcal G}:=\cup_{j=1}^{g_{R}}\cup_{i=1}^{q} P^{-1}(\delta^{j}_{i})$ divides $S$ into genus one and genus zero surfaces. 

We choose a non-dividing simple loop for each surface $T_{j} \subset R$. If we adjoin the $\pi$-liftings of all of these loops to the collection ${\mathcal G}$, then we get a collection ${\mathcal F}$ of loops which is $G$-invariant and which divides $S$ into planar surfaces. This collection of loops provides a handlebody for which $G$ extends.
\end{proof}

\begin{rema}\label{smallgroups}
As all the groups of odd order at most $241$ are abelian-by-cyclic (these include the abelian groups), Proposition \ref{proposemidirecto} asserts that every free action of such groups always extends to a handlebody.
\end{rema}

%%%%%%%%%%%%%
\subsection{A special set of loops}
In the proof of Proposition \ref{proposemidirecto}, we obtained a very particular Schottky system of loops. In the next result, we observe that this is always the case for 
a finite group $G$ of orientation-preserving homeomorphisms acting freely on a closed orientable surface $S$.

\begin{prop}\label{lemita00}
Let $G$ be a finite group of orientation-preserving homeomorphisms of a closed orientable surface $S$ which acts freely and $R=S/G$ has genus $h \geq 2$. Let $h_{0}=1$ if $h=2$ and $h_{0}=h$ if $h \geq 3$. Then
\begin{enumerate}[leftmargin=15pt]
\item $G$ extends to a handlebody if there is a collection ${\mathcal A} \subset R$ consisting of $h_{0}$ pairwise disjoint essential simple dividing loops on $R$, each one dividing $R$ into a genus one surface and other of genus $h-1$, such that each loop in the collection ${\mathcal G} \subset S$, obtained by lifting ${\mathcal A}$ to $S$, has trivial $G$-stabilizer.

\item  If $G$ is either of odd order or it has a unique element of order two (for instance, a direct product ${\mathbb Z}_{2} \times \hat{G}$ where $\hat{G}$ has odd order), then the above condition is an ``if and only if" statement.
\end{enumerate}
\end{prop}

\begin{proof}
(1) Let $\pi_{G}:S \to R$ be a Galois covering with deck group $G$, and let ${\mathcal A}=\{\delta_{1},\ldots,\delta_{h_{0}}\} \subset R$ the collection of dividing simple loops as in the hypothesis. In this case, $R \setminus {\mathcal A}$ consists of $h$ surfaces, say  $T_{1},\ldots,T_{h}$, each one of genus one with one boundary component ($T_{j}$ having as boundary the loop $\delta_{j}$), and (for $h \geq 3$) one planar surface $T_{0}$ with $h$ boundary loops (all the loops in ${\mathcal A}$). As each loop in $\pi_{G}^{-1}(\delta_{j})$  has trivial $G$-stabilizer, it follows that each connected component $X_{j,k}$ of $\pi_{G}^{-1}(T_{j})$ (for $j=1,\ldots,h$) is of genus one (with exactly $|G|$ boundary components) and (for $h \geq 3$) each connected component $Y$ of $\pi_{G}^{-1}(T_{0})$ (for $j=1,\ldots,h$) is of genus zero and $\pi_{G}:Y \to T_{0}$ is a homeomorphism.

Let $\eta_{j} \subset X_{j}$ be a non-dividing simple loop (so $X_{j} \setminus \eta_{j}$ is a planar surface with three boundary components). Then the liftings, under $\pi_{G}$, of $\eta_{j}$ will cut off each of the surfaces $X_{j,k}$ into planar surfaces.
 By adding to ${\mathcal G}=\pi_{G}^{-1}({\mathcal A})$ all of these lifted loops will provide a collection ${\mathcal F}$ of pairwise disjoint simple loops which is $G$-invariant and such that $S \setminus {\mathcal F}$ consists of planar surfaces, so $G$ extends to a handlebody.

(2) Assume $G$ extends to a handlebody.
By the Nielsen realization theorem \cite{Kerckhoff}, we may assume that $S$ is a closed Riemann surface and that $G$ is a group of its conformal automorphisms. In this case, $R=S/G$ is a closed Riemann surface of genus two and we may think of $F$ as a Fuchsian group uniformizing $R$, that is, $R={\mathbb H}^{2}/F$. If $F_{S}$ denotes the 
 kernel of $\theta$, then $S={\mathbb H}^{2}/F_{S}$ and $G=F/F_{S}$.  The extension of $G$ to a handlebody is equivalent to the existence of a 
 Schottky group $\Gamma$ of rank $g$, with a region of discontinuity $\Omega$, and a Galois covering map $P:\Omega \to S$, with deck group $\Gamma$,
for which $G$ lifts. This means that, for every element $\phi \in G$, there is a M\"obius transformation $M_{\phi} \in {\rm PSL}_{2}({\mathbb C})$ such that $P \circ M_{\phi} = \phi \circ P$. The group obtained by the liftings of all elements of $G$ is a Kleinian group $K$ which contains $\Gamma$ as a normal subgroup and $K/\Gamma=G$.
This extension property asserts the existence of a surjective homomorphism $\rho:F \to K$ and a 
surjective homomorphism $\omega:K \to G$, whose kernel is $\Gamma$, such that $\theta=\omega \circ \rho$. By Theorem \ref{picturevirtual},
 $K$ is the free product, in the sense of the Klein-Maskit combination theorems \cite{Maskit:Comb, Maskit:Comb4},  of $\alpha \geq 0$ cyclic groups generated by loxodromic transformations, and $\beta \geq 0$ groups, each one generated by a loxodromic transformation and an elliptic transformation, both of them commuting, such that $\alpha+\beta=h$.
 In any of these situations, there are simple loops in $\Omega$ over which the free product is done. Such a collection of loop projects to $R=\Omega/K$ as a collection of simple loop ${\mathcal A}$ as desired.
\end{proof}

\begin{rema}[Proposition \ref{lemita00} in the presence of orientation-reversing elements]\label{noorientable}
Note from the above proof (and by Theorem \ref{picturevirtual}) that a similar result is still valid for the case that $G$ contains orientation-reversing elements. In this more general situation, the system of loops ${\mathcal A}$ on $S/G$ has the property that each loop cuts off $R$ into two surfaces, one of them being either a torus or a Klein bottle, and the lift of any of the loops has trivial $G$-stabilizer.
\end{rema}

In the case $h=2$, Proposition \ref{lemita00} can be stated as follows.

\begin{lemm}\label{lemita}
Let $G$ be a finite group of orientation-preserving homeomorphisms of a closed orientable surface $S$ which acts freely and $S/G$ has genus two. Then 
\begin{enumerate}[leftmargin=15pt]
\item $G$ extends to a handlebody if there is an essential simple loop $\gamma \subset S$ satisfying the following properties:
\begin{enumerate}
\item the $G$-stabilizer of $\gamma$ is trivial;
\item for every $a \in G \setminus \{1\}$, $a(\gamma) \cap \gamma = \emptyset$;
\item if ${\mathcal G}=\{a(\gamma): a \in G\}$, then each connected componet of $S \setminus {\mathcal G}$ has genus one.
\end{enumerate}
\item If $G$ is either of odd order or it has a unique element of order two (for instance, a direct product ${\mathbb Z}_{2} \times \hat{G}$ where $\hat{G}$ has odd order), then the above condition is an ``if and only if" statement.
\end{enumerate}
\end{lemm}

\begin{rema}\label{divide}
As seen in the proof of part (1) of Proposition \ref{lemita00}, for $h=2$, we may assume that the loop $\gamma \subset S$ (as in Lemma \ref{lemita}) projects to $S/G$ as an essential dividing simple loop.
\end{rema}

%%%%%%%%%%%%%%%%%
\subsection{Automorphism group of the genus two fundamental group}\label{geometrico}
The fundamental group of a closed orientable surface $R$ of genus two has a presentation as follows:
$$F=\langle x_{1},x_{2},y_{1},y_{2}: [x_{1},y_{1}][x_{2},y_{2}]=1\rangle,$$
where $[a,b]=aba^{-1}b^{-1}$. Figure \ref{figura1} shows the generators of $F$ seen as simple loops in $R$.

\begin{figure}[htb]
\centering
\includegraphics[width=2.5in]{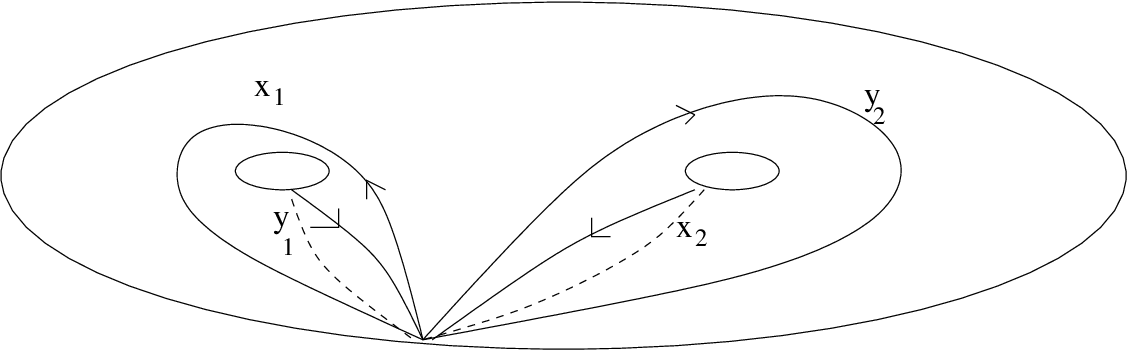}
\caption{The group $F$ seen as $\pi_{1}(R)$}
 \label{figura1}
\end{figure}

%%%%%%%%%%%%%%

By identifying $F$ with $\pi_{1}(R)$ (Figure \ref{figura1}), the Nielsen theorem \cite{Nielsen} asserts that  every automorphism of $F$ is induced by a homeomorphism of $R$. We denote by ${\rm Aut}^{+}(F)$ the subgroup consisting of those induced by orientation-preserving ones and by ${\rm Out}^{+}(F)={\rm Aut}^{+}(F)/{\rm Inn}(F)$. The group ${\rm Out}^{+}(F)$ can be identified with the mapping class group ${\rm MCG}(R)={\rm Hom}^{+}(R)/{\rm Hom}_{0}(R)$ of $R$.

Some of elements of ${\rm Aut}^{+}(R)$ are given by the automorphisms $\sigma_{j}$, which are described as Dehn-twits along simple loops $w_{j}$, where (up to isotopy) $w_{1}=y_{1}$, $w_{2}=x_{1}$, $w_{3}=x_{2}^{-1}y_{1}$, $w_{4}=y_{2}$ and $w_{5}=x_{2}$. In terms of the given generators of $F$, these automorphisms are the following ones:
\begin{equation}\label{MCG}
\begin{array}{l}
\sigma_{1}:(x_{1},y_{1},x_{2},y_{2}) \mapsto (x_{1}y_{1},y_{1},x_{2},y_{2}),\\
\sigma_{2}:(x_{1},y_{1},x_{2},y_{2}) \mapsto (x_{1},y_{1}x_{1}^{-1},x_{2},y_{2}),\\
\sigma_{3}:(x_{1},y_{1},x_{2},y_{2}) \mapsto (x_{2}^{-1}y_{1}x_{1},y_{1},x_{2},x_{2}^{-1}y_{1}y_{2}),\\
\sigma_{4}:(x_{1},y_{1},x_{2},y_{2}) \mapsto (x_{1},y_{1},x_{2}y_{2}^{-1},y_{2}),\\
\sigma_{5}:(x_{1},y_{1},x_{2},y_{2}) \mapsto (x_{1},y_{1},x_{2},y_{2}x_{2}^{-1}).
\end{array}
\end{equation}

If we set $\sigma_{5+j}=\sigma_{j}^{-1}$, then
\begin{equation}
\begin{array}{l}
\sigma_{6}:(x_{1},y_{1},x_{2},y_{2}) \mapsto (x_{1}y_{1}^{-1},y_{1},x_{2},y_{2}),\\
\sigma_{7}:(x_{1},y_{1},x_{2},y_{2}) \mapsto (x_{1},y_{1}x_{1},x_{2},y_{2}),\\
\sigma_{8}:(x_{1},y_{1},x_{2},y_{2}) \mapsto (y_{1}^{-1}x_{2}x_{1},y_{1},x_{2},y_{1}^{-1}x_{2}y_{2}),\\
\sigma_{9}:(x_{1},y_{1},x_{2},y_{2}) \mapsto (x_{1},y_{1},x_{2}y_{2},y_{2}),\\
\sigma_{10}:(x_{1},y_{1},x_{2},y_{2}) \mapsto (x_{1},y_{1},x_{2},y_{2}x_{2}).
\end{array}
\end{equation}

Let us denote by ${\rm Out}_{0}^{+}(F)$ the subgroup of ${\rm Aut}^{+}(F)$ generated by $\sigma_{1},\ldots,\sigma_{5}$.

\begin{rema}
The elements $\sigma_{1},\ldots,\sigma_{5}$ project to a set of generators of ${\rm Out}^{+}(F)$. This permitted to observe (see \cite{BM}) that ${\rm Out}^{+}(F)$ is a homomorphic image of the Artin-braid group 
$$B_{6}=\langle \sigma_{1},\ldots,\sigma_{5}:
\sigma_{i}\sigma_{j}=\sigma_{j}\sigma_{i}, \; |i-j| \geq 2, \; \sigma_{i}\sigma_{i+1}\sigma_{i}=\sigma_{i+1}\sigma_{i}\sigma_{i+1}, \; i=1,2,3,4\rangle.$$
\end{rema}

%%%%%%%%%%
\subsection{The collection ${\mathfrak C}$}\label{Sec:C}
Let us denote by ${\mathfrak C} \subset F$ the (infinite) collection of all of the images under ${\rm Out}_{0}^{+}(F)$ of the commutator $[x_{1},y_{1}]$.
The following identities:
$$
\sigma_{j}([x_{1},y_{1}])=\left\{
\begin{array}{ll}
[x_{1},y_{1}],& j \in \{1,2,4,5,6,7,9,10\},\\
(x_{2}^{-1}y_{1})[x_{1},y_{1}](x_{2}y_{1}^{-1}), & j=3,\\
(y_{1}^{-1}x_{2})[x_{1},y_{1}](y_{1}x_{2}^{-1}), & j=8,
\end{array}
\right.$$
permit to see that  ${\mathfrak C}$ consists of the elements 
$$[x_{1},y_{1}],\; (x_{2}^{-1}y_{1})[x_{1},y_{1}](x_{2}y_{1}^{-1}), \; (y_{1}^{-1}x_{2})[x_{1},y_{1}](y_{1}x_{2}^{-1}),$$
together those of the form
$$\sigma_{i_{n}}\circ \cdots \circ \sigma_{i_{1}} \circ \sigma_{l}([x_{1},y_{1}]), \; 
l \in \{3,8\}, \; n \geq 1, \; i_{j} \in \{1,\ldots,10\}.$$

\begin{rema}
Let $\theta:F \to G$ be a surjective homomorphism and set $a=\theta(x_{1})$, $b=\theta(y_{1})$, $c=\theta(x_{2})$ and $d=\theta(y_{2})$.
If $w \in {\mathfrak C} \cap \ker(\theta)$, then it produces suitable $u_{w},v_{w} \in G$ such that 
$[a,b]=[u_{w},v_{w}]$. For instance, 
(i) if $w=\sigma_{3}([x_{1},y_{1}]) \in \ker(\theta)$, then $[a,b]=[b^{-1},c]$,
(ii) if $w=\sigma_{8}([x_{1},y_{1}]) \in \ker(\theta)$, then $[a,b]=[c^{-1},b]$,
(iii) if $w=\sigma_{2}(\sigma_{3}([x_{1},y_{1}])) \in \ker(\theta)$, then $[a,b]=[ab^{-1},c]$,
(iv) if $w=\sigma_{3}(\sigma_{3}([x_{1},y_{1}])) \in \ker(\theta)$, then $[a,b]=[b^{-1},cb^{-1}c]$,
(v) if $w=\sigma_{4}(\sigma_{3}([x_{1},y_{1}])) \in \ker(\theta)$, then $[a,b]=[b^{-1},cd^{-1}]$,
(vi) if $w=\sigma_{1}(\sigma_{2}(\sigma_{3}([x_{1},y_{1}]))) \in \ker(\theta)$, then $[a,b]=[a,c]$.
\end{rema}

%%%%%%%%%%%%%%%%%
%%%%%%%%%%%%%%%%%
\section{A simple condition for extendability of free actions to handlebodies}
Let us consider a free action of a finite group $G$ on a closed orientable surface $S$ such that $R=S/G$ has genus $\gamma \geq 2$. Such an action corresponds to a 
surjective homomorphism $\theta:F^{\gamma} \to G$, where
$F^{\gamma}=\langle x_{1},y_{1},\ldots,x_{\gamma},y_{\gamma}: \prod_{j=1}^{\gamma}[x_{j},y_{j}]=1\rangle$ (which we identify with the genus $\gamma$ fundamental group of $R$).
In this section, we provide necessary and sufficient conditions for the existence of a Schottky system of loops for $G$ in terms of $\theta$. Below, we deal first with the case $\gamma=2$ (which is the case we will need later)and then the case $\gamma \geq 3$. Finally, we discuss the extendability to handlebodies in terms of a suitable set of generators for $G$ (this part will be not necessary for what follows).

%%%%%%%%%%%%
\subsection{Necessary and sufficient conditions: The case $\gamma=2$}
In this case, $F^{2}=F$.
If $a=\theta(x_{1})$, $b=\theta(y_{1})$, $c=\theta(x_{2})$ and $d=\theta(y_{2})$, then $G$ is generated by $a,b,c,d$ and it has as one of its relations the following: $[a,b][c,d]=1$. 
Let ${\mathfrak C} \subset F$ as defined in Section \ref{Sec:C}. 

\begin{theo}\label{main0}
Let $G$ be a finite group, either of odd order or with a unique element of order two, which can be generated by four elements $a,b,c,d$ such that $[a,b][c,d]=1$. 
A free action of $G$ on a surface $S$, induced by a surjective homomorphism
$\theta:F \to G$, extends to a handlebody if and only if there is some $\phi \in {\rm Out}_{0}^{+}(F)$, such that $\phi([x_{1},y_{1}]) \in \ker(\theta)$, i.e., if and only if  ${\mathfrak C} \cap \ker(\theta) \neq \emptyset$. 
\end{theo}
\begin{proof}
Let us fix a finite group $G$, realized as a group of orientation-preserving homeomorphisms of an orientable surface $S$, acting freely and with quotient $R=S/G$ of genus two (so $S$ has genus $g=g_{G}=1+|G|$). The free action is defined by a surjective homomorphism $\theta:F \to G$ (unique up to post-composition with an automorphism of $G$ and pre-composition by an automorphism of $F$). 
By Lemma \ref{lemita} (and Remark \ref{divide}), the free action of $G$ extends to a handlebody if and only if there is an essential dividing simple loop $\delta \subset R$
 that lifts to exactly $|G|$ simple loops on $S$.
For any two essential simple diving loops on $R$, say $\delta_{1}, \delta_{2}$, there is an orientation-preserving homeomorphism of $R$ carrying $\delta_{1}$ to $\delta_{2}$. By Section \ref{geometrico}, this provides the desired result for ${\rm Aut}^{+}(F)$ instead of ${\rm Out}_{0}^{+}(F)$. But, as inner automorphisms keep invariant 
the kernel of $\theta$, we may replace ${\rm Aut}^{+}(F)$ by ${\rm Out}_{0}^{+}(F)$. The last part follows from Section \ref{geometrico}.
\end{proof}

\begin{rema}
As a consequence of Theorem \ref{main0}, if $N$ is a normal subgroup of $F$, of a finite odd index, such that $N \cap {\mathfrak C} = \emptyset$, then the canonical projection  $\theta:F \to G=F/N$ provides a free action of $G$ that does not extend to a handlebody.
Given an element $w \in {\mathfrak C}$, one may use GAP \cite{GAP} to check if it belongs to ${\mathfrak C} \cap N$ (but ${\mathfrak C}$ is infinite). 
Fortunately, in $F$ there are only a finite number of subgroups of a fixed index $d=|G|$. This means that the ${\rm Aut}^{+}(F)$-orbit of $N$ is finite, that is, there is a suitable finite subcollection ${\mathfrak C}_{G}$ of ${\mathfrak C}$ such that the above free action of $G$ does not extend to a handlebody if and only if ${\mathfrak C}_{G} \cap N =\emptyset$.
\end{rema}

\begin{coro}\label{estrategia}
Let $\theta:F \to G$ be a surjective homomorphism, where $G$ is a finite group of odd order, such that $[x_{1},y_{1}] \notin \ker(\theta)$. Let $N$ be the (finite) intersection of all the ${\rm Aut}^{+}(F)$-images of $\ker(\theta)$. Then the free action of the group $G_{N}=F/N$ (of odd order) induced by the natural projection $\theta_{N}:F \to F/N$ cannot extend to a handlebody. If moreover, $B_{0}(G_{N})=0$, then this is an extendable free action that does not extend to a handlebody.
\end{coro}
\begin{proof}
Let us denote by $N_{1}=\ker(\theta), \ldots, N_{r}$ the ${\rm Aut}^{+}(F)$-orbit of $\ker(\theta)$ and let $N:=N_{1} \cap \cdots \cap N_{r}$. We claim that 
$N \cap {\mathfrak C} = \emptyset$. In fact,
let us assume, by contradiction, there is some $w \in  N \cap {\mathfrak C}$. Let $\phi_{j} \in {\rm Out}_{0}^{+}(F)$ be such that $N_{j}=\phi_{j}(N_{1})$. Then there is some $k \in \{1,\ldots,r\}$ and there is some $\phi \in {\rm Aut}^{+}(F)$ with $\phi(N_{1})=N_{1}$ such that $w=\phi_{k}(\phi([x_{1},y_{1}]))$. So, as $w \in N < N_{k}$, it holds that $\phi_{k}(\phi([x_{1},y_{1}])) \in N_{k}$, that is, $[x_{1},y_{1}] \in \phi^{-1}(\phi_{k}^{-1}(N_{k}))=N_{1}=\ker(\theta)$, a contradiction. 
Now, we may consider the canonical surjective homomorphism $\theta_{N}:F \to G_{N}=F/N$. We note that $G_{N}$ has odd order ($G_{N}$ is a subgroup of $G^{r}$). As a consequence of Theorem \ref{main0}, the corresponding free action of $G_{N}$ cannot extend to a handlebody. The last part follows from Samperton´s theorem.
\end{proof}

\begin{rema}
Corollary \ref{estrategia} enables us to create an algorithm for constructing free actions that don't extend to a handlebody (see Section \ref{Sec:main} for details).
\end{rema}

%%%%%%%%%%%%%%%
\subsection{Necessary and sufficient conditions: The case $\gamma \geq 3$}
Let ${\mathfrak C}^{\gamma}$ be the ${\rm Aut}^{+}(F^{\gamma})$-orbit of the set $\{[x_{1},y_{1}],\ldots,[x_{\gamma},y_{\gamma}]\}$, where ${\rm Aut}^{+}(F^{\gamma})$ denotes the group of automorphisms of $F^{\gamma}$ induced by orientation-preserving homeomorphisms of $R$.

\begin{theo}\label{main00}
Let $G$ be a finite group, either of odd order or with a unique element of order two, which can be generated by $2\gamma$ elements $a_{1},b_{1}\ldots,a_{\gamma},b_{\gamma}$ such that $[a_{1},b_{1}] \cdots [a_{\gamma},b_{\gamma}]=1$, where $\gamma \geq 3$. 
A free action of $G$ on a surface $S$, induced by a surjective homomorphism
$\theta:F^{\gamma} \to G$, extends to a handlebody if and only if there is some $\phi \in {\rm Aut}^{+}(F^{\gamma})$ such that ${\mathfrak C} \cap \ker(\theta) \neq \emptyset$. 
\end{theo}

The proof of the above result follows from similar arguments as for the proof of Theorem \ref{main0} (by using Proposition \ref{lemita00}). It is also possible to write down an equivalent result for the case that $G$ admits orientation-reversing elements (see Remark \ref{noorientable})

%%%%%%%%%%%%%%%
\subsection{Some extra remarks}
This section is optional and can be skipped without consequence, as it is not necessary for the rest.

Let $G$ be a finite group (either of odd order or with a unique element of order two) admitting a set of generators $a,b,c,d$ satisfying the relation $[a,b][c,d]=1$. Let $\theta:F \to G$ be a surjective homomorphism. As a consequence of Theorem \ref{main0}, if the induced free action of $G$ by $\theta$ extends to a handlebody, then ${\mathfrak C} \cap \ker(\theta) \neq \emptyset$, from which we obtain the existence of a set of generators $s_{1}, s_{2}, s_{3}, s_{4}$ such that $[s_{1},s_{2}]=1=[s_{3},s_{4}]$. The converse of this fact is provided by the following.

\begin{prop}\label{coromain0}
Let $G$ be a finite group, either of odd order or with a unique element of order two.
Then there exists a free action of $G$ on some surface $S$ such that $S/G$ has genus two and which extends to a handlebody if and only if 
$G$ can be generated by four elements $s_{1}, s_{2}, s_{3}, s_{4} \in G$ such that $[s_{1},s_{2}]=1=[s_{3},s_{4}]$.
\end{prop}

\begin{proof}
The ``only if" part is a consequence of Theorem \ref{main0} as previously stated.
Conversely, assume $G$ admits a set of generators, $s_{1},s_{2},s_{3},s_{4}$, such that $[s_{1},s_{2}]=1=[s_{3},s_{4}]$, then we may consider a Kleinian group $K$ 
\begin{equation}\label{eqK}
K=\langle A_{1}, B_{1}: B_{1}^{n_{1}}=[A_{1},B_{1}]=1 \rangle * \langle A_{2}, B_{2}: B_{1}^{n_{1}}=[A_{1},B_{1}]=1 \rangle \cong ({\mathbb Z} \times {\mathbb Z}_{n_{1}}) * ({\mathbb Z} \times {\mathbb Z}_{n_{2}}),
\end{equation}
where $n_{1}$ is the order of $s_{2}$ and $n_{2}$ is the order of $s_{4}$. In this case, we may consider the surjective homomorphism $\widehat{\omega}:K \to G$ defined by $\widehat{\omega}(A_{1})=s_{1}$, $\widehat{\omega}(B_{1})=s_{2}$, $\widehat{\omega}(A_{2})=s_{3}$ and $\widehat{\omega}(B_{2})=s_{4}$. The kernel of $\widehat{\omega}$ is a Schottky group $\widehat{\Gamma}$ of rank $g$ such that $\widehat{S}=\Omega/\widehat{\Gamma}$ admits the group $G$ as a group of automorphisms acting freely and $\widehat{S}/G$ of genus two (but it might be that this topological action of $G$ in this surface is not topologically conjugated to the one we started on $S$).
\end{proof}

\begin{rema}
The referee, in his report, provided me with another argument for the ``only if" part which is as follows. Suppose the free action of $G$ extends to a handlebody $M$ and that $G$ is either an odd order or has a unique element of order $2$.  (In other words, suppose $G$ has no dihedral subgroups and, hence, no Platonic subgroups.)  The quotient space $H=M/G$ is also a handlebod, and the action is equivalent to the data of a branched cover of $H$.  Since $G$ has no Platonic or dihedral subgroups, the branch locus is non-singular, i.e. it is an embedded link $L$ inside of $H$.  Given any separating disk $D$ of $H$ in general position with respect to $L$, we have an even number of intersection points of $D$ and each component of $L$.  We can then surgery the branched cover in a regular neighborhood $N(D) = D \times (-1,1)$ of $D$ to ``re-wire'' the intersection points in pairs in such a way that the new branch locus does not intersect $D = D \times \{0\}$.  This permits us to observe that there is a new handlebody $M'$ (whose boundary is $S$) for which the action of $G$ extends and from which one can obtain the desired result. 
\end{rema}

\begin{rema}
Below, we provide some direct consequences of Proposition \ref{coromain0}.
\begin{enumerate}[leftmargin=15pt]

\item Proposition\ref{coromain0} can be restated as follows. The free action of $G$ on $S$ extends to a handlebody if and only if 
there is a set of generators $\{S_{1}, S_{2}, S_{3}, S_{4} \}$ of $F$ such that: (i) each $S_{1}, S_{2}, S_{3}, S_{4}, [S_{1},S_{2}]$ represents an essential simple loop,  (ii) $[S_{1},S_{2}][S_{3},S_{4}]=1$, and (iii) $[S_{1},S_{2}]  \in \ker(\theta)$. A known algorithm to detect elements of $F$ representing simple loops is provided in \cite{BS} but it seems to be hard to check.

\item As a consequence of Proposition \ref{coromain0}, we may observe the following. Let $G$ a finite group of odd order which: (i) it can be generated by elements $a,b,c,d$ such that $[a,b][c,d]=1$, and (ii) it does not have a set of generators $s_{1},s_{2},s_{3},s_{4}$ such that  $[s_{1},s_{2}]=1=[s_{3},s_{4}]$.
Then every free action of $G$ on a surface of genus $g_{G}=1+|G|$ never extends to a handlebody. In particular, if moreover $B_{0}(G) =0$, then every free action of $G$ in genus $g_{G}$ is extendable, but not to a handlebody.

\item Two groups of orientation-preserving homeomorphisms of a surface $S$ are topologically equivalent if an orientation-preserving homeomorphism of $S$ conjugates one into the other. We say that $G$ is topologically rigid in genus $1+|G|$ if any two free actions of $G$ in that genus are topologically equivalent. In this setting, the above provides the following. Let $G$ be a finite group of odd order. 
If the topological free action of $G$ in genus $1+|G|$ is rigid, then the non-extendability of $G$ to a handlebody is equivalent to the non-existence of a set of generators $s_{1},s_{2},s_{3},s_{4}$ of $G$ with $[s_{1},s_{2}]=1=[s_{3},s_{4}]$.

\end{enumerate}
\end{rema}

%%%%%%%%%%%%%%%%%%%
%%%%%%%%%%%%%%%%%%
\section{Finite groups acting freely which extend but not to handlebodies}\label{Sec:main}
In this section, we first provide a theoretical algorithm that permits the construction of free actions that extend but not to handlebodies. Then, we use it to obtain a negative answer to our original question.

%%%%%%%%%%%%%%%%%%%
\subsection{An algorithm to construct free actions that extend but not to handlebodies}
Corollary \ref{estrategia} permits us to obtain the following algorithm to create free actions that extend but not to a handlebody.

\begin{theo}\label{algoritmo}
Let $G$ be a finite non-abelian group of odd order such that all of its Sylow subgroups are abelian, let $\theta:F=\langle x_{1},y_{1},x_{2},y_{2}: [x_{1},y_{1}][x_{2},y_{2}]=1\rangle \to G$ be a surjective homomorphism such that $[x_{1},y_{1}] \notin \ker(\theta)$, and let $N$ be the intersection of the (finite) collection of all the ${\rm Aut}^{+}(F)$-images of $\ker(\theta)$. 
If $G_{N}=F/N$ and $S_{N}$ is the surface defined by the normal subgroup $N$, then the free action of $G_{N}$ on $S_{N}$ extends but not to handlebody.
\end{theo}

\begin{rema}
The lowest order of a group $G$ satisfying the hypothesis of Theorem \ref{algoritmo} is $21$.
\end{rema}

Note that Theorem \ref{algoritmo} provides examples of free actions that extend but not to a handlebody once we have a free action of a non-abelian group of odd order whose Sylow subgroups are all abelian. The existence of such actions is given below.

\begin{coro}\label{main2}
There are finite groups of orientation-preserving homeomorphisms of a closed orientable surface that act freely and such that they extend but not to a handlebody.
\end{coro}
\begin{proof}
The semi-direct product 
$$G=\langle a,b: a^{7}=1=b^{3}, \; bab^{-1}=a^{2} \rangle \cong {\mathbb Z}_{7} \rtimes {\mathbb Z}_{3},$$
is a group of odd order ($|G|=21$) whose Sylow subgroups are cyclic groups of order $3$ and $7$.
Let us consider the surjective homomorphism 
$\theta:F \to G$, defined by 
$$\theta(x_{1})=a,\; \theta(y_{1})=b,\; \theta(x_{2})=c=b^{-1}, \; \theta(y_{2})=d=ba^{-1}.$$

We observe that $[x_{1},y_{1}] \notin \ker(\theta)$. We may now apply our algorithm (Theorem \ref{algoritmo}) to obtain a free action of a group $G_{N}$ that extends but not to a handlebody. 
\end{proof}

\begin{rema}
Unfortunately, for the group $G\cong {\mathbb Z}_{7} \rtimes {\mathbb Z}_{3}$ and $\theta:F \to G$ (as in the above proof), we have not been able to compute the corresponding group $G_{N}$ that is produced by our algorithm.

\end{rema}

%%%%%%%%%%%%%%%%%%
\subsection{Proof of Theorem \ref{algoritmo}}
As a consequence of Corollary \ref{estrategia}, we already know that $N \cap {\mathfrak C}=\emptyset$, that is, 
the free action of $G_{N}$, as a group of orientation-preserving homeomorphisms of $S_{N}$, does not extend to a handlebody. To finish the proof, we only need to check that $B_{0}(G_{N})=0$. For it, 
we proceed to describe $S_{N}$ and $G_{N}$ in terms of a suitable fiber product.

The normal subgroup $N_{1}=\ker(\theta)$ is a normal subgroup of $F$ of finite index $|G|=d \geq 3$ odd.
 As $F$ has a finite number of subgroups of index $d$, the ${\rm Aut}^{+}(F)$-orbit of $N_{1}$ consists of a finite collection $N_{1}, \ldots, N_{r}$ of normal subgroups of $F$. Let us consider the intersection $N=N_{1} \cap \cdots \cap N_{r}$; a finite index normal subgroup which is invariant under ${\rm Aut}^{+}(F)$. For each $j \in \{1,\ldots,r\}$, the natural projection $\theta_{j}:F \to G_{j}=F/N_{j}$ induces a free action of $G \cong G_{j}$ on a closed orientable surface $S_{j}$ such that $S_{j}/G_{j}=R$
(as before, $R$ represents a fixed closed orientable surface of genus two). In this case, $\theta_{1}=\theta$ and $G_{1}=G$.

Let $\pi_{j}:S_{j} \to R$ be a Galois covering with deck group $G_{j}$ associated with the free action induced by $\theta_{j}$. The fiber product of these $r$ pairs $(S_{1},\pi_{1}), \ldots, (S_{r},\pi_{r})$,
$$\hat{S}=\{(s_{1},\ldots,s_{r}) \in S_{1} \times \cdots \times S_{r}: \pi_{1}(s_{1})=\cdots=\pi_{r}(s_{r})\} \subset S_{1} \times \cdots \times S_{r},$$
is a (possibly non-connected) closed orientable surface. This surface admits the group $G^{r} = G_{1} \times \cdots \times G_{r}$ as a group of orientation-preserving homeomorphisms, acting freely and with $\hat{S}/G^{r}=R$. Moreover, the map $\pi:\hat{S} \to R$ defined by $\pi(s_{1},\ldots,s_{r})=\pi_{1}(s_{1})(=\pi_{j}(s_{j}))$ is a Galois covering with deck group $G^{r}$. The natural projections $P_{j}:\hat{S} \to S_{j}$, defined by $P_{j}(s_{1},\ldots,s_{r})=s_{j}$ are Galois coverings, with deck group $G^{r-1}$ (obtained from $G^{r}$ by eliminating the factor $G_{j}$), satisfying $\pi=\pi_{j} \circ P_{j}$.

The surface $\hat{S}$ might not be connected, but any two connected components of $\hat{S}$ are necessarily homeomorphic, and each of these connected components is defined by the group $N$
\cite{HRV}. In this way, we may assume (up to biholomorphisms) that $S_{N}$ is one of these connected components. Also, if $H$ denotes the 
$G^{r}$-stabilizer of $S_{N}$ (note that, $H$ has odd order as $G$ has odd order), then $H$ is a group of orientation-preserving homeomorphisms of $S_{N}$ which acts freely and 
such that $S_{N}/H=R$. So, we may identify $G_{N}$ with $H$.

Now, as a $p$-Sylow subgroup of $G^{r}$ is a direct product of $r$ copies of $p$-Sylow subgroups of $G$ (which are abelian by our hypothesis), the $p$-subgroups of the subgroup $G_{N}$ are also abelian. It follows that $B_{0}(G_{N})=0$ as desired.

%%%%%%%%%%%%%%%%%%
%%%%%%%%%%%%%%%%%%%
\section{Examples of handlebody extensions of free actions}\label{Sec:ejemplos}
In Theorem \ref{main0} (see also Theorem \ref{main00}), we have provided a necessary and sufficient condition for the extendability to a handlebody of 
a given free action of an odd order group $G$ acting freely on a closed orientable surface $S$ of genus two such that $R=S/G$ has genus two. More precisely, if the free action is induced by 
a surjective homomorphism $\theta:F \to G$, where $F=\langle x_{1},x_{2},y_{1},y_{2}: [x_{1},y_{1}][x_{2},y_{2}]=1\rangle$, then it extends to a handlebody if and only if
${\mathfrak C} \cap \ker(\theta) \neq \emptyset$. One can try to use GAP to determine if the previous intersection is not empty. However, as the collection ${\mathfrak C}$ is infinite, computations are challenging. We know that there is a suitable finite subcollection ${\mathfrak C}_{G}$, depending on $G$, such that ${\mathfrak C} \cap \ker(\theta) \neq \emptyset$ if and only if ${\mathfrak C}_{G} \cap \ker(\theta) \neq \emptyset$. The problem is that we do not know how to detect the correct finite subset ${\mathfrak C}_{G}$ to perform our checking.

In this section, we will provide some examples and, for each one, we check if the above intersection is non-empty, that is, if they extend to handlebodies. For our computations, we have 
considered the finite subcollection 
${\mathfrak C}_{0}$ (of cardinality $13.446$) consisting of those elements of ${\mathfrak C}$ of the form 
$$[x_{1},y_{1}],\; (x_{2}^{-1}y_{1})[x_{1},y_{1}](x_{2}y_{1}^{-1}), $$
together with those of the form
$$\sigma_{i_{n}}\circ \cdots \circ \sigma_{i_{1}} \circ \sigma_{3}([x_{1},y_{1}]), \; 
n \in \{1,2,3,4,5,6,7,8,9\}, \; i_{j} \in \{1,\ldots,5\},$$
and used GAP to compute ${\mathfrak C}_{0} \cap \ker(\theta)$.

%%%%%%%%%%
\subsection{Example 1}\label{Sec:Ejemplo1}
Our first collection of examples is given by free actions by certain abelian-by-cyclic groups. As we know, from Theorem \ref{proposemidirecto}, each of these actions extends to a handlebody. 

Let (i) $3 \leq q < p$ be prime integers,  (ii) $2 \leq r<p$, and (iii) $r^{q} \equiv 1 \mod(p)$ and the semi-direct product  
$$G=\langle a,b: a^{p}=1=b^{q}, \; bab^{-1}=a^{r} \rangle \cong {\mathbb Z}_{p} \rtimes {\mathbb Z}_{q},$$

As $p$ and $q$ are relatively prime, then $B_{0}(G)=0$ (see \cite{Kang} and Section \ref{semidirecto}).
Note that 
$$[a,b]=a^{1-r}, \; [b^{-1},ba^{-1}]=a^{r-1}.$$

Let us consider the surjective homomorphism 
$\theta:F \to G$, defined by 
$$\theta(x_{1})=a,\; \theta(y_{1})=b,\; \theta(x_{2})=c=b^{-1}, \; \theta(y_{2})=d=ba^{-1}.$$

We observe that $[x_{1},y_{1}] \notin \ker(\theta)$.
As the free action induced by $\theta$ extends to a handlebody, we know (by Theorem \ref{main0}) that ${\mathfrak C} \cap \ker(\theta) \neq \emptyset$. We proceed to check ${\mathfrak C}_{0} \cap \ker(\theta)$ in many of these cases.

First, let $w_{0}:=\sigma_{5}(\sigma_{4}(\sigma_{3}([x_{1},y_{1}])=y_{2}x_{2}^{-2}y_{1} [x_{1},y_{1}] x_{2}^{2}y_{2}^{-1}y_{1}^{-1}\in {\mathfrak C}$. We observe that $w_{0}$  belongs to  $\ker(\theta)$ if and only if $r^{3} \equiv 1 \mod(p)$. As $r^{q} \equiv 1 \mod(p)$ and $r \in \{2,\ldots,p-1\}$, we obtain that this is equivalent to have $q=3$ (as $q$ is a prime integer) and $r^{3} \equiv 1 \mod(p)$. This asserts that, for the tuples 
of the form $[p,3,r]$, one has that $w_{0} \in \ker(\theta)$. 

If $5 \leq q <p \leq 23$, our computations permited to obtain that ${\mathfrak C}_{0} \cap \ker(\theta) \neq \emptyset$ (see table \ref{tabla1}).

\begin{center}
\small{
\begin{figure}\label{tabla1}
\begin{tabular}{|c|c|}
\hline 
[p,q,r] & $w \in {\mathfrak C}_{0} \cap \ker(\theta)$    \\\hline 
[ 11, 5, 3 ] & $(y_{2}^{2}x_{2}^{-1}y_{1})^{2}x_{1}y_{1}x_{1}^{-1}(y_{1}^{-1}x_{2}y_{2}^{-2})^2y_{1}^{-1}$     \\\hline    
[ 11, 5, 4 ] & $x_{2}^{-1}y_{1}y_{2}x_{2}^{-1}(x_{2}^{-1}y_{1})^2x_{1}y_{1}x_{1}^{-1}(y_{1}^{-1}x_{2})^2x_{2}y_{2}^{-1}y_{1}^{-1}x_{2}y_{1}^{-1}$   \\\hline    
[ 11, 5, 5 ] & $x_{2}^{-1}x_{1}^{-2}y_{1}^{-1}x_{2}x_{1}y_{1}x_{1}$                                  \\\hline    
[ 11, 5, 9 ] & $x_{2}^{-1}(y_{1}x_{1}^{-1})^2y_{1}^{-1}x_{2}x_{1}^{2}y_{1}^{-1}$               \\\hline  
[ 23, 11, 2 ] &  $(y_{2}x_{2}^{-1}y_{1})^2x_{1}y_{1}x_{1}^{-1}(y_{1}^{-1}x_{2}y_{2}^{-1})^2y_{1}^{-1}$            \\\hline   
[ 23, 11, 3 ] & $y_{2}^{2}x_{2}^{-1}y_{1}x_{1}y_{1}x_{1}^{-1}y_{1}^{-1}x_{2}y_{2}^{-2}y_{1}^{-1}$                   \\\hline   
[ 23, 11,4 ] & $x_{2}^{-1}y_{1}y_{2}x_{2}^{-1}y_{1}^{2}x_{1}^{-1}(y_{1}^{-1}x_{2})^{2}y_{2}^{-1}x_{1}y_{1}^{-1}$       \\\hline   
[ 23, 11, 6 ] & $x_{2}^{-1}y_{1}x_{1}^{-1}x_{2}^{-1}y_{1}^{2}x_{1}^{-1}(y_{1}^{-1}x_{2}x_{1})^{2}y_{1}^{-1}$ \\\hline 
[ 23, 11, 8 ] & $x_{2}^{-1}y_{1}y_{2}x_{2}^{-1}y_{1}x_{1}^{-1}y_{2}x_{2}^{-1}y_{1}^{2}x_{1}^{-1}y_{1}^{-1}x_{2}(y_{1}^{-1}x_{2}y_{2}^{-1}x_{1})^{2}y_{1}^{-1}$ \\\hline 
[ 23, 11, 9 ] & $y_{2}x_{2}^{-1}y_{1}x_{1}^{-1}y_{2}x_{2}^{-1}y_{1}^{2}x_{1}^{-1}(y_{1}^{-1}x_{2}y_{2}^{-1}x_{1})^{2}y_{1}^{-1}$ \\\hline 
[ 23, 11, 12 ] & $(y_{2}x_{2}^{-1})^{2}x_{2}^{-1}y_{1}^{2}x_{1}^{-1}y_{1}^{-1}x_{2}(x_{2}y_{2}^{-1})^{2}x_{1}y_{1}^{-1}$ \\\hline 
[ 23, 11, 13 ] & $x_{2}^{-1}y_{1}x_{1}^{-1}y_{2}x_{2}^{-1}y_{1}x_{1}^{-1}x_{2}^{-1}y_{1}^{2}x_{1}^{-1}y_{1}^{-1}x_{2}x_{1}y_{1}^{-1}x_{2}y_{2}^{-1}x_{1}y_{1}^{-1}x_{2}x_{1}y_{1}^{-1}$ \\\hline 
[ 23, 11, 16 ] & $y_{2}x_{2}^{-1}y_{1}y_{2}^2x_{2}^{-1}(y_{2}x_{2}^{-1}y_{1})^{2}x_{1}y_{1}x_{1}^{-1}(y_{1}^{-1}x_{2}y_{2}^{-1})^{2}x_{2}y_{2}^{-2}y_{1}^{-1}x_{2}y_{2}^{-1}y_{1}^{-1}$ \\\hline 
[ 23, 11, 18 ] & $(y_{2}x_{2}^{-1})^{2}x_{2}^{-1}y_{1}x_{1}y_{1}x_{1}^{-1}y_{1}^{-1}x_{2}(x_{2}y_{2}^{-1})^{2}y_{1}^{-1}$ \\\hline 
\end{tabular}
\caption{Tuples $[p,q,r]$, where $5 \leq q <p \leq 23$, in Example \ref{Sec:Ejemplo1} with a corresponding $w \in {\mathfrak C}_{0} \cap \ker(\theta)$}
\end{figure}
}
\end{center}

\begin{rema}
Let us note that these examples can be used in our algorithm (Theorem \ref{algoritmo}) to obtain more examples of free actions that extend but not to a handlebody.
\end{rema}

%%%%%%%%%%%%%%%
\subsection{Example 2}
Let  $G$ be the group generated by $\alpha,\beta,\gamma,\delta,$
together the following relations
$$\alpha^{3},\beta^{3},\gamma^{3},\delta^{3}, [\alpha,\beta]^{3}, [\alpha,\delta], [\gamma,\delta], [[\alpha,\beta],\alpha], [[\alpha,\beta],\beta], $$
$$\gamma^{-1}\alpha \gamma\beta^{-1}\alpha^{-1}\beta, \gamma^{-1}\beta \gamma\alpha^{-1}\beta^{-1}\alpha, \delta^{-1}\beta \delta\alpha^{-1}\beta^{-1}\alpha.$$
$$G = \langle \alpha,\beta \rangle \rtimes \langle \gamma,\delta\rangle \cong 
({\mathbb Z}_{3}^{2} \rtimes {\mathbb Z}_{3}) \rtimes  {\mathbb Z}_{3}^{2},$$
If we set $a=\alpha$, $b=\gamma$, $c=\beta$ and $d=\delta$, then $[a,b][c,d]=1$. 
In the GAP Library, this group corresponds to ${\rm SmallGroup}(3^{5},65)$ and (using GAP)  $B_{0}(G)=0$, so every free action of $G$ always extends. Next, we observe that all of those free actions (on genus $244$) necessarily extend to hadlebodies.

\begin{prop}
Every free action of  the group ${\rm SmallGroup}(3^{5},65)$ on genus $244$ extends to a handlebody.
\end{prop}
\begin{proof}
Let $\theta:F \to G$ be any surjective homomorphism and set $u=\theta(x_{1})$ and $v=\theta(y_{1})$. If $[u,v]=1$, then (by Theorem \ref{main0}) the free action of $G$ induced by $\theta$ extends to a handlebody. Let us assume now that $[u,v]\neq 1$. By computations with GAP, we may observe that, up to ${\rm Aut}(G)$, there is only one pair $(u,v) \in G \times G$ such that $[u,v] \neq 1$.  So, up to post-composition by an automorphism of $G$, we only need to consider those $\theta$ such that $\theta(x_{1})=a$ and $\theta(y_{1})=b$. If we set $r=\theta(x_{2})$ and $t=\theta(y_{2})$, then the pair $(r,t)$ satisfies the following
\begin{enumerate}
\item $[a,b][r,t]=1$,
\item  $G=\langle a,b,r,t\rangle$;
\end{enumerate}

As the elements in the set $\{x_{1}x_{2}, ,x_{1}y_{1}^{\pm 1},x_{1}y_{2}^{-1}, x_{2}y_{2}^{\pm 1}, x_{2}y_{1}^{-1}, y_{2}y_{1}\} \subset F$ represent essential simple loops, if   
$1 \in \{ur, uv^{\pm 1},ut^{-1}, rt^{\pm 1}, rv^{-1},tv\}$, the free action of $G$ induced by $\theta$ agains extends to a handlebody.

Let us now consider those pairs $(r,t)$ such that $1 \notin \{ur, uv^{\pm 1},ut^{-1}, rt^{\pm 1}, rv^{-1},tv\}$. Such a collection of pairs has cardinality $12.312$. Using GAP, for each of these pairs $(r,t)$ we were able to check that ${\mathfrak C}_{0} \cap \ker(\theta) \neq \emptyset$ and, again by Theorem \ref{main0}, such a free action extends to a handlebody.
\end{proof}

%%%%%%%%%%%%%%%
\subsection{Example 3}
Let us consider the group $G={\rm SmallGroup}(3^{5},28)$ which, by \cite{Samperton}, admits some free action on genus $244$ which does not extend. 

There are many tuples $(a,b,r,t)$ of elements of $G$ satisfying that $[a,b][r,t]=1$ and $G=\langle a,b,r,t\rangle$. 
For each such tuple $(a,b,r,t)$, let us consider the 
surjective homomorphisms
$\theta_{a,b,r,t}:F \to G$, defined by $\theta_{a,b,r,t}(x_{1})=a$, $\theta_{a,b,r,t}(y_{1})=b$, $\theta_{a,b,r,t}(x_{2})=r$ and $\theta_{a,b,r,t}(y_{2})=t$.

If $[a,b]=1$, then Theorem \ref{main0} asserts that the induced free action extends to a handlebody. 
So, let us assume, from now on, that $[a,b] \neq 1$.

The list of pairs 
$(a,b) \in G \times G$ with the property that $[a,b] \neq 1$ has cardinality $54.432$. Up to ${\rm Aut}(G)$, there are $96$ such pairs. For computations, we have fixed a particular pair $(a,b)$. For such fixed pair, we obtained $3.132$ other pairs $(r,t)$ such that (i) $[a,b][r,t]=1$ and $G=\langle a,b,r,t\rangle$. 
Next, for each of these $3.132$ tuples $(a,b,r,t)$, we used GAP to compute ${\mathfrak C}_{0} \cap \ker(\theta_{a,b,r,t})$ and obtained that only $1.188$ of them have non-empty intersections (so they extend to a handlebody). Note that there may be other possible outcomes that could occur for the cases for which ${\mathfrak C}_{0} \cap \ker(\theta_{a,b,r,t})=\emptyset$ (we may still have ${\mathfrak C} \cap \ker(\theta_{a,b,r,t}) \neq \emptyset$).

%%%%%%%%%%%%%%%
\subsection{Example 4}
Let $p \geq 3$ be a prime and consider the Heisenberg group (of order $p^{3}$)
$$G:=\langle x,y: x^{p}=y^{p}=[x,y]^{p}=[x,[x,y]]=[y,[x,y]]=1\rangle.$$

As $G$ is an extraspecial group, $B_{0}(G)=0$, so every free action of $G$ extends. Let us consider the case $p=3$. In this case, the list of pairs 
$(u,v) \in G \times G$ with the property that $[u,v] \neq 1$ has cardinality $432$ and, up to ${\rm Aut}(G)$, there is only one. One of these pairs is $(a,b)=(x,y)$. For such a pair, we obtain $189$ other pairs $(r,t)$ such that (i) $[a,b][r,t]=1$ and $G=\langle a,b,r,t\rangle$. For each $(r,t)$, we consider the surjective homomorphism $\theta_{a,b,r,t}:F \to G$, defined by $\theta_{a,b,r,t}(x_{1})=a$, $\theta_{a,b,r,t}(y_{1})=b$, $\theta_{a,b,r,t}(x_{2})=r$ and $\theta_{a,b,r,t}(y_{2})=t$. For each of them, we use GAP to check that ${\mathfrak C}_{0} \cap \ker(\theta_{a,b,r,t}) \neq \emptyset$. So, all of these free actions of $G$  extend to a handlebody.

%%%%%%%%%%%%%%%%%%%%%%
%%%%%%%%%%%%%%%%%%%%%
\section{Final Remark}
The following question was raised by the referee:

{\it What exactly is the smallest example of one of these groups that admits an extendable action (with a quotient of genus $2$) that does not extend to a handlebody?}
 
In Remark \ref{smallgroups}, we observed that every free action of a group of odd order at most $241$ always extends to a handlebody. So, the next odd order to search for an example, of a group that extends but not to handlebody, is $3^5=243$. In this case, we have $67$ non-isomorphic groups. From them, there are $50$ which are abelian-by-cyclic (so every free action extends to a handlebody). For the other $17$ groups, there are exactly three which admit free actions which do not extend (by Samperton's result) and, for all the others $14$, every free action extends (not necessarily to a handlebody). Anyway, all of these $17$ groups can be generated by four elements $a,b,c,d$ such that $[a,b]=1=[c,d]$, so each of them admits a free action that extends to a handlebody. In order to apply our algorithm, to these special examples, we need to know when to stop our searching (as ${\mathfrak C}$ is infinite); which is not known to the author.

We have provided an algorithm to produce free actions that extend but not to a handlebody. To produce such examples, we only need a free action $\theta:F \to G$, of a finite non-abelian group $G$ of odd order whose Sylow subgroups are abelian, such that $[x_{1},y_{1}] \notin \ker(\theta)$. Our algorithm makes use of a suitable fiber product construction to provide us with a finite group $G_{N}$ acting freely which extends but not to a handlebody. Unfortunately, to obtain $G_{N}$ we need to consider the intersection of all ${\rm Aut}^{+}(F)$-images of $\ker(\theta)$; which again it seems to be computationally expensive.

In particular, we have not been able to answer the above question.
In future work, we plan to find explicit examples to answer the aforementioned question.

%%%%%%%%%%%%%%%%%
\subsection*{Acknowledgements}
The author is indebted to E. Samperton and A. Carocca, for valuable discussions on preliminary versions of this paper. 
I want to express my gratitude to the referee for providing me with valuable feedback, suggestions, and corrections, and for an alternative proof of the ``only if" part of Proposition \ref{coromain0} which does not make use of Kleinian groups.

%%%%%%%%%%%%%%%%%%
%%%%%%%%%%%%%%%%%%

\end{document}